\documentclass[reqno, 12pt]{amsart}%
\usepackage{amsaddr}
\usepackage{amssymb}
\usepackage[nesting]{hyperref}
\usepackage[pdftex]{graphicx}
\usepackage{listings}
\usepackage{multirow}
\usepackage{placeins}
\usepackage{color}
\usepackage{subfigure}
\usepackage{lscape}
\usepackage{dsfont}
\usepackage{amsmath}
\usepackage{amsfonts}
\usepackage{tikz}
\usepackage{verbatim}
\usepackage{accents} 
\usepackage{todonotes}
\usepackage{mathtools}
\usepackage{bbm}
\usepackage{algorithm}
\usepackage{algpseudocode}
\usepackage{subcaption}
\usepackage{todonotes}
\usepackage{natbib}  

\newcommand{\BlackBoxes}{\global\overfullrule5pt}

\BlackBoxes

\providecommand{\U}[1]{\protect\rule{.1in}{.1in}}

\allowdisplaybreaks

\newcommand{\R}{\mathbb{R}} 
\newcommand{\N}{\mathbb{N}} 

\newcommand{\PP}{\mathbb{P}}
\newcommand{\EE}{\mathbb{E}}

\DeclareMathOperator*{\argmax}{argmax}

\newtheorem{theorem}{Theorem}

\newtheorem{proposition}[theorem]{Proposition}
\theoremstyle{definition}
\newtheorem{example}[theorem]{Example}
\newtheorem{remark}[theorem]{Remark}

\numberwithin{equation}{section} \numberwithin{theorem}{section}
\def\0{\kern0pt\-\nobreak\hskip0pt\relax}

\makeatletter
\AtBeginDocument{ \def\@serieslogo{ \vbox to\headheight{ \parindent\z@ \fontsize{6}{7\p@}\selectfont
\vss}}}

\def\makeoverbar#1#2#3#4#5#6#7{ \setbox0=\hbox{$\m@th#2\mkern#5mu{{}#3{}}\mkern#6mu$} 
\setbox1=\null \dimen@=#4\fontdimen8#13 \dimen@=3.5\dimen@
\advance\dimen@ by \ht0 \dimen@=-#7\dimen@ \advance\dimen@ by \wd0
\ht1=\ht0 \dp1=\dp0 \wd1=\dimen@
\dimen@=\fontdimen8#13 \fontdimen8#13=#4\fontdimen8#13
\rlap{\hbox to \wd0{$\m@th\hss#2{\overline{\box1}}\mkern#5mu$}}
\fontdimen8#13=\dimen@}
\def\mylabel#1#2{{\def\@currentlabel{#2}\label{#1}}}
\makeatother
\begin{document}

\renewcommand*{\arraystretch}{1.5}

\begin{center}
{\Large\bf Yet Another Distributional Bellman Equation}
\end{center}

\begin{center}
{\bf \large  Nicole B\"auerle$^{a}$\footnote{Corresponding author}, Tamara G\"oll$^{a}$, Anna Ja\'skiewicz$^{b}$ }
\end{center}
\vspace*{0.5cm}
\noindent$^{a}$Department of Mathematics, Karlsruhe Institute of Technology, 
Karlsruhe, Germany,
{\footnotesize  email: {\it nicole.baeuerle@kit.edu}, {\it tamara.goell@kit.edu}}\\

\noindent$^{b}$Faculty of Pure and Applied Mathematics, Wroc{\l}aw University of Science and Technology,  Wroc\l aw, Poland,
{\footnotesize  email: {\it anna.jaskiewicz@pwr.edu.pl}}\\

\begin{center}
 \today
\end{center}

    \\

\noindent
{\bf Abstract. }We consider non-standard Markov Decision Processes (MDPs) where the target function is not only a simple expectation of the accumulated reward. Instead, we consider rather general functionals of the joint distribution of terminal state and accumulated reward which have to be optimized. For finite state and compact action space, we show how to solve these problems by defining a lifted MDP whose state space is the space of distributions over the true states of the process. We derive a Bellman equation in this setting, which can be considered as a distributional Bellman equation. Well-known cases like the standard MDP and quantile MDPs are shown to be special examples of our framework. We also apply our model to a variant of an optimal transport problem. 

\section{Introduction}
Classical MDP theory is concerned with the maximization of expected accumulated reward $\EE R_N$ over $N$ discrete stages or an infinite time horizon where the transition law of the process can be controlled. Thus, we are interested in the first moment of the random variable $R_N$ representing the accumulated reward and aim to maximize this. For the standard theory, solving these problems with the help of a Bellman equation, a recursive equation for the optimal value of the problem depending on the time horizon, see \cite{puterman2014markov,lerma96,br11}. 

However, later there has been increased interest in criteria which also account for risk in the sense of deviation from the mean or in higher moments of the accumulated reward. Hence, in criteria which address other aspects of the underlying return distribution. A widely discussed criterion for example is the risk-sensitive reward $-\EE \exp(-\lambda R_N)$ which has first been addressed by \cite{howard1972risk,jaquette1976utility}. It can be considered as a criterion taking all moments of the accumulated reward into account since
$$ -\EE \exp(-\lambda R_N) = - \sum_{k=0}^\infty \frac{(-\lambda)^k\EE(R_N^k)}{k!}.$$
Risk-sensitive Markov decision problems can again be solved using a Bellman equation, but the value function is time-dependent even in the stationary infinite horizon setting (see the aforementioned references). For Q-learning in this setting, see \cite{borkar2002q,mihatsch2002risk,borkar2010learning,shen2014risk}.
Further extensions include more general expected utility (\cite{chung1987discounted,bauerle2014more}), the application of risk measures (\cite{ruszczynski2010risk,bauerle2011markov,shen2013risk,bauerle2021minimizing,moghimi2025beyond}) or the consideration of mean-variance problems (\cite{mannor2011mean,cui2014unified,bauerle2024time,bauerle2025mean}). In the first two cases, it is again possible to solve the problem with a Bellman equation on an extended state space, where an auxiliary variable like the 'accumulated reward so far' has to be added to the natural state of the process. For mean-variance problems, one can take advantage of the fact that the variance can be represented as an optimization problem. A recent overview can be found in \cite{bauerle2024markov}. Besides this, other papers considered quantile optimization in the sense that $\PP(R_N \ge t)$ for fixed $t\in\R$ has to be optimized, see e.g. \cite{filar1995percentile,wu1999minimizing,chow2015risk,gilbert2017optimizing,li2022quantile}. In this case, the problem can again be solved using a state space extension. In particular, the latter three papers are also concerned with finding efficient algorithms to solve MDPs with quantile criteria. In \cite{marthe2024beyond} the authors discuss among others which utility problems may be solved by dynamic programming. 


On the other hand, there is a stream of literature which deals with distributional reinforcement learning. In this theory, the aim is to learn the distribution of the accumulated returns under a stationary Markovian policy (\cite{bellemare2017distributional,bdr2023,rowland2019statistics}) or equivalently to learn the quantile function (\cite{dabney2018distributional,dabney2018implicit}) and use parametric families to achieve this efficiently. This is similar to Q-learning for a fixed policy. Optimization enters the scene by choosing greedy actions according to an appropriate functional which relates to the target quantity. However, in general this theory applies to problems where the classical MDP theory can be used (maybe on an extended state space) and where the optimal policy can be found within the deterministic stationary policies (\cite{lyle2019comparative} compares traditional RL (reinforcement learning) to distributional RL). 

In this paper, we combine the distributional approach with a very general optimization target involving the joint distribution of current state and accumulated reward. More precisely, if $F_N^\sigma$ is the joint distribution of the accumulated reward $R_{N-1}$ and current state $X_N$ under policy $\sigma$, we are interested in maximizing $H(F_N^\sigma)$ where $H$ maps distributions to real values. Hence, $H$ could map the distribution on the expectation of $R_{N-1}$, yielding a classical MDP, to a quantile, to a probability or to the distance to a given distribution. This gives a very flexible target, comprising, in principle, all previously considered cases.  Also situations with constraints may be handled (see e.g. \cite{altman2021constrained, basak2001value}). Of course, and this is crucial to note, we cannot expect that we can restrict the search for optimal policies to Markovian policies, nor can we restrict to deterministic policies in general. However, we can show that optimal policies depend only on the current distribution of state and accumulated reward so far. Thus, they depend on the history of the process, but only through a certain sufficient statistic. We can also identify cases where the optimal decision rule is not randomized. This is achieved by introducing a 'lifted' MDP with states given by distributions. The approach is inspired by recent papers considering  MDPs with distributions as states (see e.g. \cite{bauerle2024time}). In order to keep the exposition simple, we restrict here to finite state and action spaces and in Sec. \ref{sec:special} to finite state and compact action space. The recent paper \cite{pires2025optimizing} considers a related question, however is concentrating on infinite horizon discounted reward. They consider a possibly infinite state space, but a finite action space. In their main theorem, the target function $H$ has to satisfy some indifference properties and they do not define a 'lifted' MDP, hence there is no typical value function. However, they provide a large number of different applications ranging from Conditional Value at Risk optimization to Deep $\eta$ networks. 

The {\em contributions} of our paper are: (i) We introduce in a proper mathematical way a very general optimization problem involving the  joint distribution of the accumulated reward  and current state of a Markov Decision process. (ii) We introduce a 'lifted' MDP to solve the problem recursively. We do this for finite state and finite/compact action spaces. (iii) We show that cases previously treated in the literature like classical MDP and quantile MDP are included as special cases.   (iv) For the infinite horizon problem, we discuss existence and reduction of optimal policies. (v) We give a new application of approximating a given distribution by a random walk which we solve in a naive dynamic programming fashion. This is a modification of an optimal mass transportation problem.

The {\em outline} of the paper is as follows: In the next section we introduce the model with finite  state and action spaces and derive the recursive solution algorithm under a continuity assumption on the objective mapping $H$. In Sec. \ref{sec:special} we consider the case of a compact action space but restrict to distributions of the terminal state plus expected reward. In order to obtain optimal policies, we need a further continuity property of the transition law. In Sec. \ref{sec:infi} we briefly discuss the infinite horizon problem in our setting. Using an example, we show that optimal stationary policies cannot be expected. However, we prove the existence of optimal policies when the target function is Wasserstein-Lipschitz and can reduce them to a certain set of policies. In the last section we consider different applications: We show that the usual Bellman equation is included for the classical objective of maximizing the expected accumulated reward. We also consider a quantile MDP where we show a similar result. The last application is non-standard: We  determine the transition probabilities of a Markov chain in such a way that, at a fixed time point, a weighted criterion of distance to a  given distribution and expected transportation cost is minimized. We solve this problem using naive dynamic programming and discuss the results.

\section{The model}\label{sec:model}

By $\R$ ($\N$) we denote the set of all real 
numbers (positive integers). Let $Y$ be a Borel space, i.e., a Borel subset of a complete separable metric space with its Borel $\sigma$-algebra 
$\mathcal{B}(Y).$ Then, $P(Y)$  stands for  the set of all probability distributions on  
$(Y,\mathcal{B}(Y)).$ We equip $P(Y)$ with the weak topology, that is, the coarsest topology for which the mapping $\mu\to\int_Y fd\mu$ is continuous for every bounded and continuous function $f: Y\to\R.$

Suppose we are given a classical Markov decision model with finite state and action space. We are first interested in models with finite time horizon $N\in\N.$ More precisely, the model is described by the following items:
\begin{itemize}
\item[(i)] $E$  is the finite state space,
\item[(ii)] $A$
 is the finite action space,
 \item[(iii)] $r:E\times A \to \R$ is the  one-stage reward,
 \item[(iv)] $q$ is the transition probability from $E\times A$ to $E,$
 \item[(v)] $g:E\to\R$ is the terminal reward.
\end{itemize}
 Thus,
let $(\Omega, \mathcal{F})$   be the measurable space with 
$\Omega= E\times (A\times  E)^N$ and 
	$\mathcal{F}$ the corresponding power set. We define the set of all histories by
 $H_n:= E\times (A\times E)^n,$ i.e., $h_n=(x_0,a_0,\ldots ,x_n)$ (for $n=1,\ldots,N)$ describes the sequence of states and actions which have occurred up to time $n.$
 For $n=0$, we have $H_0=E$. The state process is then given by the random variables $X_n:\Omega \to E$ with $X_n(h_N)=x_n$ and the action process  by the random variables $A_n:\Omega \to A$ with $A_n(h_N)=a_n.$    A policy $\sigma$ is a sequence
 $(\sigma_n)_{n=0}^{N-1}$ of history dependent decision rules, where each decision rule  $\sigma_n $ specifies the probability distribution $\sigma_n(\cdot| h_n)$ on the action space $A$ for
 $h_n \in H_n$ and $n=0,\ldots, N-1.$
We denote this set of policies by $\Pi_N.$  
It is well-known that  a  policy 
$\sigma=(\sigma_n)_{n=0}^{N-1}\in \Pi_N,$ an  initial distribution $\nu$ and the transition probability $q$ induce a  probability distribution $\PP_\nu^\sigma$ on $(\Omega,\mathcal{ F})$ 
(see p. 23 in \cite{puterman2014markov}) so that 
 \begin{eqnarray*}
 && \PP_\nu^\sigma(X_0 =x_0, A_0=a_0, X_1=x_1,\ldots ,X_N=x_N)=\\
  & & \nu(x_0) \sigma_0(a_0|x_0) q(x_1|x_0,a_0) \sigma_1(a_1|x_0,a_0,x_1) \cdot\ldots \cdot \sigma_{N-1}(a_{N-1}|h_{N-1}) q(x_{N}|x_{N-1},a_{N-1}).
    \end{eqnarray*}
In classical MDP theory, we are interested in maximizing the expected reward of the system over the time horizon when the initial distribution is $\nu$. Thus, we define  for $n=0,1,\ldots ,N-1$ 
$$ R_n := \sum_{k=0}^n r(X_k,A_k) \quad \mbox{and} \quad R_N := \sum_{k=0}^{N-1} r(X_k,A_k)+g(X_N).$$
For a fixed initial distribution $\nu$, one then tries to solve
$$ V(\nu) = \sup_{\sigma\in \Pi_N} \EE_{\nu}^\sigma[R_N]$$
where $\EE_{\nu}^\sigma$ denotes the expectation w.r.t.\ $\PP_{\nu}^\sigma$.
It is well-known that the optimal policy (which exists here, since state and action spaces are finite) can be found among the deterministic Markovian decision rules and that it does not depend on $\nu$. 
More precisely, when we define $V_N(x):= g(x)$ and for $n=N-1,\ldots,0$
\begin{equation}\label{eq:OEclassic0}
    V_n(x)= \max_{a\in A} \big\{r(x,a) +
    \sum_{x'}  V_{n+1}(x')q(x'|x,a)\big\},
\end{equation} 
then $V(\nu)= \sum_{x} V_0(x) \nu(x)$ and the maximizers in \eqref{eq:OEclassic0} define an optimal (deterministic) policy.

In this paper however, we generalize the 
optimization criterion to one where the joint 
distribution of the accumulated reward and terminal 
state is involved. Since state and action spaces 
are finite, the random variables $R_0,\ldots, R_N$ 
take only a finite number of possible values. In 
what follows, we denote by  $S$ the finite set of 
all possible realizations of the random variables 
$R_0,R_1,\ldots ,R_N$. 
For a fixed policy $\sigma\in \Pi_N$, let  $F_n^{\sigma}$ be the joint distribution of $(X_{n},R_{n-1})$, i.e.,
$$ F_n^\sigma(x,s) = \PP_\nu^\sigma(X_{n}=x,R_{n-1}=s), \quad (x,s)\in E\times S, \ n\ge 1.$$ 
Obviously, this distribution depends on the chosen policy $\sigma$ (and the initial distribution $\nu$ which is fixed and thus not part of the notation).  Let now $H:P(E\times S) \to 
\R$ be an arbitrary functional. The aim is to solve
the following optimization problem
\begin{equation}\label{eq:prob1}
    \sup_{\sigma\in \Pi_N} H(F_N^\sigma).
\end{equation} 
The solution in (\ref{eq:prob1})
 is called {\it a value function.}  \\

\begin{example} \label{exam1}
 Let us consider some special cases of our setting. Let $F\in P(E\times S).$
\begin{itemize}
\item[a)] Of course the classical case is 
included by defining 
$$H(F)=\sum_{x,s} (g(x)+s) F(x,s).$$
Obviously $H(F_N^\sigma)= \EE_\nu^\sigma[R_N] $ in 
this case and it is just the expected accumulated 
reward of the system and we can use the standard 
Bellman equation \eqref{eq:OEclassic0} to solve the 
problem.
\item[b)] When we choose 
$H(F)=  \sum_{x,s} F(x,s) \chi(x,s),$ where
 $$ \chi(x,s):= \left\{ \begin{array}{cl}
    1 , & \mbox{ if } s+g(x)\ge t  \\
    0 , & \mbox{ else }
 \end{array}\right.$$ for a fixed $t\in\R$, then  $H(F_N^\sigma)= \PP^\sigma_\nu(R_{N}\ge t) $ is the quantile of the terminal accumulated reward. By varying the function $\chi$ we can treat the probability that the accumulated reward ends up in different areas.
\item[c)]   Let $E\subset \R$. If $W_1$ is the Wasserstein distance between the distributions, then we might consider
$$H(F) = W_1(F(\cdot,S),G)=  \int_{\R} |F_c(t)-G_c(t)|dt,$$
where $G\in P(\R)$ is the given target distribution. Here, $F_c$ and $G_c$ denote the corresponding cumulative distributions.
Then the aim is to choose the policy 
$\sigma$ such that the distribution of $X_N$ is as 
close as possible to  
$G$, i.e., we wish to find $\inf_{\sigma\in\Pi_N} H(F^\sigma_N).$ The distance $W_1$ may be replaced by any other reasonable distance between distributions. 
\end{itemize}
\end{example}

\noindent 
In order to solve problems like \eqref{eq:prob1}, we have to lift the MDP to a more general state space which is given by joint distributions of the accumulated reward and original state of the process. Moreover, actions in the lifted MDP are transition kernels from the augmented state $(x,s)$ of the process to the action space.  Thus, we define
$$ \Pi^M := \{ \pi : E\times S \to P(A) \},$$
where $S$ is as before the finite set of all possible realizations of the accumulated rewards.
More precisely, we formally define the {\em lifted MDP} which is a deterministic dynamic control problem by the following data:

\begin{itemize}
\item[(i)] $P(E\times S)$  is the  state space  where the interpretation of a  state $F_n\in P(E\times S)$ at time point $n$ is given as the joint distribution of $(X_{n},R_{n-1})$. Note here that since state and action spaces are finite, $F_n$ is ultimately a discrete distribution on a finite number of points.
\item[(ii)] $ \Pi^M$
 is the action space.
 \item[(iii)] The  one-stage reward is zero.
 \item[(iv)] $H:P(E\times S) \to \R$ is the  terminal reward.
 \item[(v)] The transition function $T:P(E\times S) \times \Pi^M \to P(E\times S)$  is given by 
\begin{align}\label{eq:Toperator}
    T^\pi(F)(x',s') :=  T(F,\pi)(x',s')=\sum_{(x,s,a) \, : \, r(x,a)=s'-s}  q(x'|x,a)\pi(a|x,s) F(x,s),
\end{align} 
where $(x',s') \in E\times S$ and $q$ and $r$ are the data from our initial MDP defined at the beginning of this section.
\end{itemize}
The previous  data defines a lifted MDP which is indeed a {\em deterministic} dynamic control problem. Policies in this lifted model are denoted by $(f_n)$ where $f_n : P(E\times S) \to \Pi^M. $ Though state and action spaces are finite in the original formulation, this is no longer true in the lifted MDP.  Indeed, the price we have to pay for the model to be deterministic now is that the state and action spaces consist of probability distributions and transition kernels, respectively. But in the end this also implies that there is no randomization over the lifted action space necessary in order to improve the value.

We first prove the following crucial connection between the original MDP and the lifted MDP. 
 The notation $\nu\otimes\mu $ stands 
for the product measure between the two probability 
distributions $\mu,\nu.$ 

\begin{proposition}\label{prop:1}
For every policy $\sigma=(\sigma_n)_{n=0}^{N-1}$ in
the original MDP model, there exists an action
sequence $(\pi_0,\ldots,\pi_{N-1})$ in the lifted 
MDP such that
$$\PP^\sigma_\nu(X_N=x, R_{N-1}=s) = 
T^{\pi_{N-1}}\circ T^{\pi_{N-2}} \circ \ldots \circ 
T^{\pi_0}(F_0)(x,s),\quad (x,s)\in E\times S,  $$
where $F_0= \nu \otimes \delta_0$ and $T^\pi$ is 
the operator defined in \eqref{eq:Toperator}.
\end{proposition}

\begin{proof}
The proof is by induction over the length of the 
time horizon. First consider $N=1.$ Then
\begin{align*}
\PP^\sigma_\nu(X_1=x', r(X_0,A_0)=s') = 
\sum_{(x,a)\; :\; r(x,a)=s'} q(x'|x,a) 
\sigma_0(a|x) \nu(x) = T^{\pi_0}(F_0)(x',s')  
\end{align*}
where $\pi_0(a|x,s)=\pi_0(a|x,0) := \sigma_0(a|x).  $
Now suppose the statement is correct up to time 
$n \le N-1.$ For abbreviation we denote 
$F_n^{(\pi_0,\ldots,\pi_{n-1})} = T^{\pi_{n-1}}\circ T^{\pi_{n-2}} \circ 
\ldots \circ T^{\pi_0}(F_0)$ for all $n\le N-1.$ 
We have to show that
$\PP^\sigma_\nu( X_{n+1}=x, R_{n}=s)= 
F_{n+1}^{(\pi_0,\ldots,\pi_{n})}(x,s)$. 
    Thus, we obtain by using the induction hypothesis in the second equation:
    \begin{align*}
      &\PP^\sigma_\nu(X_{n+1}=x', R_n=s') = \\
    &  \sum_{x,s}  \PP^\sigma_\nu(X_{n+1}=x', R_n=s' | X_n=x, R_{n-1}=s) \PP^\sigma_\nu( X_n=x, R_{n-1}=s) = \\
      & \sum_{x,s}  \PP^\sigma_\nu(X_{n+1}=x', R_n=s' | X_n=x, R_{n-1}=s) F_n^{(\pi_0,\ldots,\pi_{n-1})}(x,s)=\\
    & \sum_{(x,s,a)  \; :\; r(x,a)=s'-s}   q(x'|x,a) \PP^\sigma_\nu(A_{n}=a | X_n=x, R_{n-1}=s) F_n^{(\pi_0,\ldots,\pi_{n-1})}(x,s)=\\
      &  \sum_{(x,s,a) \; :\; r(x,a)=s'-s}q(x'|x,a)\pi_n(a|x,s) F_n^{(\pi_0,\ldots,\pi_{n-1})}(x,s)  =  T^{\pi_n}(F_n^{(\pi_0,\ldots,\pi_{n-1})})(x',s'),
    \end{align*}
   where we define $ \pi_n(a|x,s) := \PP^\sigma_\nu(A_{n}=a | X_n=x, R_{n-1}=s) $. Obviously, the definition of $\pi$ depends on $\sigma.$
   This proves the statement.
\end{proof}

\begin{remark} \label{rem23}
    Proposition \ref{prop:1} may appear surprising at first sight, since there is an arbitrary (history-dependent) policy on the left-hand side and a 'Markov' policy (in the sense of the lifted MDP) on the right hand side. However, it is well-known that in Markov decision processes, for any history dependent policy, we can always find  a Markov policy such that the distribution of state and action at every time point (i.e.,\ the marginal distributions) coincide, see e.g. Theorem 2 in \cite{derman1966note} or Corollary~2.1 in \cite{kallenberg2002finite}. This is essentially what also happens in Proposition 2.1.
\end{remark}

 On the other hand, given any action sequence $(\pi_n)_{n=0}^{N-1}$ it is immediate (since the accumulated reward up to time $n-1$ is a function of the history up to that time)  that by choosing $\sigma_n(\cdot|h_n) := \pi_n(\cdot|x_n,s_{n-1}), $ we can construct from any sequence $(\pi_n)_{n=0}^{N-1}$ a sequence $(\sigma_n)_{n=0}^{N-1}$ such that the equality in Proposition \ref{prop:1} holds. As a result, we can write problem \eqref{eq:prob1} as follows:
\begin{equation}\label{prob:II}
    \sup_{\sigma\in \Pi_N} H(F_N^\sigma) = \sup_{(\pi_0,\ldots, \pi_{N-1})} H\big(T^{\pi_{N-1}}\circ T^{\pi_{N-2}} \circ \ldots \circ T^{\pi_0}(F_0)\big).
\end{equation} 
The right-hand side in (\ref{prob:II}) is a 
deterministic dynamic control problem, which can be 
solved with the help of the value functions $J_n$, 
where 
\begin{align}
  \nonumber  J_N(F) & := H(F),\\ \label{eq:OE}
    J_n(F) &:= \sup_{\pi\in \Pi^M} J_{n+1}(T^\pi(F)),\quad n=0,1,\ldots ,N-1, 
\end{align}
for $F\in P(E\times S).$
We make the following assumption:
\begin{description}
\item[(C1)] The mapping $F\to H(F)$ is upper 
semicontinuous, that is,
$\limsup_{k\to\infty} H(F_{(k)}) \le H(F)$
for any sequence $(F_{(k)})$ converging to $F$ in
$P(E\times S)$ as $k\to\infty,$ i.e., 
$F_{(k)}(x,s)\to F(x,s)$ for every $(x,s)\in E\times S.$
\end{description}

\noindent
Note that this condition is satisfied for all cases in  Example \ref{exam1}.

We obtain the following result:

\begin{theorem}\label{theo:finite_main}
Assume (C1), i.e., $H$ is upper semicontinuous. 
\begin{itemize}
    \item[a)] If $(J_n)$ is computed according to \eqref{eq:OE}, then $J_0(F_0)$
    is the value function of problem \eqref{eq:prob1}, i.e.,
    $J_0(F_0)= \sup_{\sigma\in \Pi_N} H(F_N^\sigma).$
    \item[b)] Maximizers $(f_0^*,\ldots,f_{N-1}^*)$ in the recursion of \eqref{eq:OE} exist, i.e., for $n=0,1,\ldots,N-1$
    $$ f_n^*(F) = \argmax_\pi J_{n+1}(T^\pi(F)),\; F \in P(E\times S).$$
    \item[c)] Define the following state-action sequence starting with state $F_0$:
    \begin{align*}
    \pi_0^* &= f_0^*(F_0),\\
    F_1 &= T^{\pi_0^*}(F_0),\\
    \vdots &\\
    \pi_n^* &= f_n^*(F_n),\\
    F_{n+1} &= T^{\pi_{n}^*}\circ \ldots \circ T^{\pi_0^*}(F_0).
\end{align*}
Then, $(\pi_0^*,\pi_1^*,\ldots,\pi_{N-1}^*)$  determines  an optimal policy $(\sigma_0^*,\ldots, \sigma_{N-1}^*)$ for problem \eqref{prob:II} as follows $$\sigma_n^*(a|h_n)= \pi_n^*\Big(a\Big|x_n, \sum_{k=0}^{n-1}r(x_k,a_k)\Big)$$ for $n=0,1,\ldots,N-1$, where $h_n=(x_0,a_0,\ldots,x_n).$
\end{itemize}

\end{theorem}

\begin{proof} The proof proceeds by backward induction. For part a) we prove that $(J_n)$  computed according to \eqref{eq:OE} satisfies
$$ J_n(F)= \sup_{(\pi_n,\ldots,\pi_{N-1})} H\big(T^{\pi_{N-1}}\circ T^{\pi_{N-2}} \circ \ldots \circ T^{\pi_n}(F)\big)$$
for $n=0,\ldots,N-1.$ For $n=N-1$, the statement follows by definition. Suppose the statement is true for $n+1.$ We prove that it is also true for $n$. By \eqref{eq:OE} and the induction hypothesis we obtain
\begin{align*}
    J_n(F) & =  \sup_{\pi_n\in \Pi^M} J_{n+1}(T^{\pi_n}(F)) \\
    & =  \sup_{\pi_n\in \Pi^M}  \sup_{(\pi_{n+1},\ldots,\pi_{N-1})} H\big(T^{\pi_{N-1}}\circ T^{\pi_{N-2}} \circ \ldots \circ T^{\pi_{n+1}}(T^{\pi_n}(F))\big) \\
    &= \sup_{(\pi_n,\ldots,\pi_{N-1})} H\big(T^{\pi_{N-1}}\circ T^{\pi_{N-2}} \circ \ldots \circ T^{\pi_n}(F)\big)
\end{align*}
which proves the statement. It remains to show that the maximum points exist. We do this again by induction, proving that all $J_n$ are upper semicontinuous and the maximum points exist. The argument for the induction starting at time point $N-1$ is essentially the same as for the induction step and we thus skip it. 
Set $m=|E|\cdot|S|.$
Consider the optimization problem in \eqref{eq:OE} and assume that $J_{n+1}: \R^{m}\to \R$ is upper semicontinuous. Let $\pi\in \Pi^M$.    
Clearly, if 
$F_{(k)}\to F$ in $P(E\times S)$ and 
$\pi_{(k)}\to \pi$ in $\Pi^M$ (i.e.,
$\pi_{(k)}(a|x,s)\to \pi(a|x,s)$ for every $a\in A$ and  $(x,s)\in E\times S$), then
    $$   \sum_{(x,s,a) \, : \, r(x,a)=s'-s}  q(x'|x,a)\pi_{(k)}(a|x,s) 
    F_{(k)}(x,s)\to \sum_{(x,s,a) \, : \, r(x,a)=s'-s}  q(x'|x,a)\pi(a|x,s) 
    F(x,s)$$
for every $(x',s')\in E\times S$ as $k\to\infty.$  
    Further, the set $\Pi^M$ is compact, since $\Pi^M=P(A)^m.$ 
    By Proposition 2.4.3 in \cite{br11} 
    there exists  $\pi^*\in\Pi^M$ such that
    \begin{align*}
        &  \sup_{\pi\in \Pi^M} J_{n+1} \Big(\sum_{(x,s,a) \, : \, r(x,a)=\cdot-s}  q(\cdot|x,a)\pi(a|x,s) F(x,s)\Big)\\
        & = J_{n+1} \Big(\sum_{(x,s,a) \, : \, r(x,a)=\cdot-s}  q(\cdot|x,a)\pi^*(a|x,s) F(x,s)\Big)
    \end{align*}
    and the mapping 
     $$  F \to  \sup_{\pi\in \Pi^M} J_{n+1} \Big(\sum_{(x,s,a) \, : \, r(x,a)=\cdot-s}  q(\cdot|x,a)\pi(a|x,s) F(x,s)\Big)$$
     is upper semicontinuous, which proves the statement. Note that the connection between $\pi^*$ and $\sigma^*$ follows from the proof of Proposition \ref{prop:1}.
\end{proof}

We call the last equation in \eqref{eq:OE} 
a 'distributional Bellman equation'. The term 'distributional Bellman equation' has 
been used before (see e.g.\ \cite{bdr2023}). However, there the term refers to a recursive computation of the cumulated reward distribution under a fixed policy, i.e.,\ in a first step there is no optimization involved.


\begin{remark}\label{rem:onlyX}
    Sometimes it may be sufficient to consider the distribution of $X_N$ only, without $R_{N-1}.$ In this case,  the transition function simplifies as follows:
\begin{align*}
    T^\pi(F)(x',S) &:= \sum_{(x,s,a) }  q(x'|x,a)\pi(a|x,s) F(x,s)\\
    &= \sum_{(x,a) }  q(x'|x,a) \sum_s \pi(a|x,s) F(x,s)\\
    &=  \sum_{(x,a) }  q(x'|x,a)  \tilde\pi(a|x) F(x,S),
\end{align*} 
 where $$ \tilde \pi(a|x) := \frac{\sum_s \pi(a|x,s) F(x,s)}{\sum_s  F(x,s)}. $$ Note that $\tilde \pi(a|x)$ is a probability distribution on the action space, since $\tilde \pi(a|x) \ge 0$ and obviously $\sum_a \tilde \pi(a|x)=1.$  
 Moreover, $\tilde \pi(a|x)$ depends only on $x$.  Hence, as the states in the lifted MDP it is  sufficient to consider distributions $F\in P(E)$.
\end{remark}


 \section{A special continuous case} \label{sec:special}
In this section we briefly discuss the case of 
a compact (not necessarily finite) Borel 
action space where we restrict ourselves to 
objective functions which depend only on the 
law of $X_N$ and add the expected reward up to 
time $N-1$. 
As noted in Remark \ref{rem:onlyX}, in this case it is sufficient to define the state of the process at time $n$ in the lifted MDP as the distribution of $X_n$.  More precisely, what we change w.r.t.\ the model in the previous section is that
\begin{description}
\item[(I)] The action set $A$ is a compact metric space.
\end{description}
Any given distribution $\nu$ for the initial state and a policy $\sigma\in\Pi_N$  
define a unique probability measure 
$\PP_\nu^\sigma$  over the space of trajectories of
the states and actions. The construction
is standard, see \cite{puterman2014markov} and the beginning of the previous section. 
We define
$$ F_n^\sigma(x) := \PP_\nu^\sigma(X_{n}=x), \quad x\in E.$$  Formally, the aim is to solve
\begin{equation}\label{eq:prob+}
    \sup_{\sigma\in \Pi_N} \Big\{ H(F_N^\sigma) +  \EE_\nu^\sigma [R_{N-1}] \Big\} 
\end{equation} 
for a function $H:P(E) \to \R.$
Thus, we obtain that the lifted MDP is defined by
\begin{itemize}
\item[(i)] $P(E)$  is the  state space, where the interpretation of a  state $F_n\in P(E)$ at time point $n$ is given as the  distribution of $X_{n}$. Note here that  $F_n$ is ultimately a discrete distribution on the finite set $E$.  In particular we have $F_0=\nu.$
\item[(ii)] $ \Pi^M = \{\pi:E \to P(A)\}$
 is the action space.
 \item[(iii)] The  one-stage reward is given by $\hat r : P(E)\times \Pi^M \to \R$ with $$\hat r (F,\pi) := \sum_x \int_A r(x,a) \pi(da|x) F(x).$$
 \item[(iv)] $H:P(E) \to \R$ is the  terminal reward.
 \item[(v)] The transition function $T:P(E) \times \Pi^M \to P(E)$  is given by 
\begin{align}\label{eq:Toperator2}
    T^\pi(F)(x') :=  T(F,\pi)(x')=\sum_{x} \int_A  q(x'|x,a)\pi(da|x) F(x)
\end{align} 
where $x' \in E$.
\end{itemize}

Note that a Markov policy in the original model 
can be linked to a sequence of actions in the 
lifted MDP (analogously to Prop. \ref{prop:1}). As mentioned in Remark \ref{rem23}
it is well-known that  
for every policy $\sigma$ there exists a Markov policy $\pi$ that induces the same marginal probability measure (see Lemma 2 in \cite{piun1997} for general models with no necessarily  finite action set). Therefore, we conclude that
for
$\sigma=(\sigma_n)_{n=0}^{N-1}$ in the original MDP model, there exists an action sequence 
$\pi=(\pi_n)_{n=0}^{N-1}$ in the lifted MDP such that
  $$F_N^\sigma(x)=\PP^\sigma_\nu(X_N=x) = 
T^{\pi_{N-1}}\circ T^{\pi_{N-2}} \circ \ldots \circ 
T^{\pi_0}(F_0)(x),\quad x\in E. $$ 
Thus, we have that
\begin{eqnarray}\label{pirat}
 && H(F_N^\sigma) +  \EE_\nu^\sigma [R_{N-1}]  \\
 &=&  H \Big( T^{\pi_{N-1}}\circ T^{\pi_{N-2}} \circ \ldots \circ T^{\pi_0}(F_0)\Big)+ \sum_{k=0}^{N-1} \hat r\Big(  T^{\pi_{k-1}}\circ  \ldots \circ T^{\pi_0}(F_0),\pi_k\Big) \nonumber
\end{eqnarray}
where $T^{\pi_{-1}}(F_0):= F_0.$ 
Thus, the value iteration has the following form:
\begin{align}
  \nonumber  J_N(F) & := H(F),\\
    J_n(F) &:= \sup_{\pi\in \Pi^M} \left\{ \hat r (F,\pi)+ J_{n+1}(T^\pi(F)) \right\},\quad n=0,1,\ldots ,N-1. \label{eq:DPEcompact}
\end{align}

In this case, we need another condition to ensure the existence of optimal policies:
\begin{description}
\item[(C1')] The mapping $F\to H(F)$ is upper 
semicontinuous, that is,
$\limsup_{k\to\infty} H(F_{(k)}) \le H(F)$
for any sequence $(F_{(k)})$ converging to $F$ in
$P(E)$ as $k\to\infty.$
    \item[(C2)] The mappings $a\to q(x'|x,a)$ and $a\to r(x,a)$ are continuous for all $x,x'\in E$. 
\end{description}

Note that 
$M:=\max_{(x,a)\in E\times A}|r(x,a)|<\infty.$
Then we obtain:

\begin{theorem}\label{theo:finite_spec}
Assume (C1') and (C2). 
\begin{itemize}
    \item[a)] If $(J_n)$ is computed according to \eqref{eq:DPEcompact} with $T^\pi$ as defined in \eqref{eq:Toperator2}, then $J_0(F_0)$ is the value function of problem \eqref{eq:prob+}, i.e.,
    $J_0(F_0)= \sup_{\sigma\in \Pi_N} \big\{ H(F_N^\sigma) +  \EE_\nu^\sigma [R_{N-1}] \}$.
    \item[b)] Maximizers $(f_0^*,\ldots,f_{N-1}^*)$ in the recursion of \eqref{eq:DPEcompact} exist, i.e., for $n=0,1,\ldots,N-1$
    $$ f_n^*(F) = \argmax_\pi  \left\{ \hat r (F,\pi)+ J_{n+1}(T^\pi(F)) \right\},\; F \in P(E).$$
    \item[c)] Define the following state-action sequence starting with state $F_0=\nu$:
    \begin{align*}
    \pi_0^* &= f_0^*(F_0),\\
    F_1 &= T^{\pi_0^*}(F_0),\\
    \vdots &\\
    \pi_n^* &= f_n^*(F_n),\\
    F_{n+1} &= T^{\pi_{n}^*}\circ \ldots \circ T^{\pi_0^*}(F_0).
\end{align*}
Then  $(\pi_0^*,\pi_1^*,\ldots,\pi_{N-1}^*)$  
determines  an optimal policy $(\sigma_0^*,\ldots, 
\sigma_{N-1}^*)$ for problem \eqref{eq:prob+} as 
follows $$\sigma_n^*(a|h_n)= \pi_n^*\big(a\big|
x_n\big)$$ for $n=0,1,\ldots,N-1$, where $h_n=(x_0,a_0,\ldots,x_n).$
\end{itemize}
\end{theorem}

\begin{proof}
Part a) can be deduced by backward induction. 
Indeed note that
\begin{eqnarray*}
J_{N-1}(F)&=&\sup_{\pi_{N-1}}\Big\{\hat r(F,\pi_{N-1})+H(T^{\pi_{N-1}}(F))\Big\}\quad\mbox{and}\\
J_{N-2}(F)&=&\sup_{\pi_{N-2}}\Big\{\hat r(F,\pi_{N-2})+J_{N-1}(T^{\pi_{N-2}}(F))\Big\}\\
&=&\sup_{(\pi_{N-2},\pi_{N-1})}\Big\{\hat r(F,\pi_{N-2})+
\hat r(T^{\pi_{N-2}}(F),\pi_{N-1})+ H(T^{\pi_{N-1}}\circ T^{\pi_{N-2}}(F))\Big\}.
\end{eqnarray*}
Consequently,
\begin{eqnarray*}
J_{N-n}(F)
&=&\sup_{(\pi_{N-n},\ldots,\pi_{N-1})}\Big\{\hat r(F,\pi_{N-n})+ \sum_{k=1}^{n-1}
\hat r(T^{\pi_{N-(k+1)}}\circ\ldots\circ T^{\pi_{N-n}}(F),\pi_{N-k})\\&&+ H(T^{\pi_{N-1}}\circ \ldots\circ T^{\pi_{N-n}}(F))\Big\} \quad\mbox{and}\\
J_0(F)&=&\sup_{(\pi_{0},\ldots,\pi_{N-1})}\Big\{\hat r(F,\pi_{0})+ \sum_{k=1}^{N-1}
\hat r(T^{\pi_{N-(k+1)}}\circ\ldots\circ T^{\pi_{0}}(F),\pi_{N-k})\\&&+ H(T^{\pi_{N-1}}\circ \ldots\circ T^{\pi_{0}}(F))\Big\}.
\end{eqnarray*}
Now observe that the expression in the curly brackets equals the right-hand side in (\ref{pirat}). This proves  part a).
For parts b) and c) we proceed as follows. 
 Let $(F_{(k)},\pi_{(k)}) \to (F,\pi)$ in $P(E)\times \Pi^M$ 
as $k\to \infty.$ Recall that it means that
$F_{(k)}(x)\to F(x)$  and 
$\pi_{(k)}(\cdot|x)\to\pi(\cdot|x)$ weakly in $P(A)$ for every $x\in X.$ We show that 
$$\hat r(F_{(k)},\pi_{(k)}) \to \hat
r(F,\pi)\quad\mbox{and}\quad T(F_{(k)},\pi_{(k)}) \to 
T(F,\pi)$$
as $k\to\infty.$
Indeed, notice that
 \begin{align*}
\lefteqn{\left| \hat r(F_{(k)},\pi_{(k)}) -
\hat r(F,\pi)\right| =}\\
 &  \left| \sum_{x} \int_A  r(x,a)
 \pi_{(k)}(da|x) F_{(k)}(x) - \sum_{x}\int_A r(x,a)
 \pi(da|x) F(x) \right| \\&
 \le  \sum_{x}  \left| \int_A r(x,a)
 \pi_{(k)}(da|x)  - \int_A  r(x,a)\pi(da|x) 
 \right| F_{(k)}(x) \\
 & +  \sum_{x} \int_A  r(x,a)\pi(da|x) 
 \left|F_{(k)}(x) - F(x)  \right|\\&
 \le  \sum_{x}  \left| \int_A r(x,a)
 \pi_{(k)}(da|x)  - \int_A  r(x,a)\pi(da|x) 
 \right|  \\
 & +  M \sum_{x} 
 \left|F_{(k)}(x) - F(x)  \right|. 
      \end{align*}
Hence, the first term tends to zero by (C2) and
the second term converges to zero by definition. 
In the same manner, we show that  
$T(F_{(k)},\pi_{(k)}) \to T(F,\pi).$
Clearly, $\pi^M$ is compact, since $\Pi^M=P(A)\times\cdots\times P(A)$ ($|E|$ times). Hence, the remaining parts follow as before from Proposition 2.4.3 in \cite{br11}. 
\end{proof}

\section{Infinite Horizon Problems}\label{sec:infi}
So far we have considered problems with a finite time horizon. Let us now briefly turn to the situation with an infinite time horizon in the setting of Section 
\ref{sec:model}. 
Hence, we deal with the same problem as in Section \ref{sec:model} with the difference that we introduce a discount factor $\beta\in(0,1)$
and set $N=\infty.$ A policy $\sigma$ is a sequence 
$(\sigma_0,\sigma_1,\ldots)$ of history dependent decision rules 
$\sigma_n : H_n \to P(A)$. By $\Pi_\infty$ we denote the set of all policies.  The random reward we are interested in is now given by
$$ R_\infty := \lim_{n\to\infty} R_n,\quad \mbox{ where}\quad R_n :=\sum_{k=0}^n \beta^k r(X_k,A_k).$$
Clearly, $R_\infty$ is well-defined and $R_\infty\in[-M/(1-\beta),M/(1-\beta)],$  where  that $M=\max_{(x,a)\in E\times A}|r(x,a)|$.
Let us consider the problem of 
maximizing the expected discounted reward when the 
initial distribution is $\nu$, that is, 
$$ \sup_{\sigma\in \Pi_\infty} \EE_{\nu}^\sigma[R_\infty].$$
As in the previous section $\EE_{\nu}^\sigma$ stands
for the expectation operator that refers to the probability measure $\PP_{\nu}^\sigma$ defined uniquely on $H_\infty:=(E\times A)^{\N}$ with 
$\sigma$-algebra generated by the cylinder sets,
see \cite{puterman2014markov}.
It is well-known that the problem 
can be solved with the help of a fixed point equation and an optimal policy can be found among stationary deterministic Markovian policies, i.e., an optimal policy 
$\sigma^*$ is of the form  $\sigma^*=(\tilde \sigma, \tilde\sigma,\ldots)$, 
where $\tilde \sigma: E\to A.$ Now for $\sigma\in \Pi_\infty$ and 
$B\in \mathcal{B}(\R)$ define 
$$F^\sigma_\infty(B)= \PP_\nu^\sigma (R_\infty \in B)$$
as the distribution of $R_\infty$ under policy $\sigma,$  when the initial distribution is given by $\nu$.
Assume that $H:P(\R)\to\R$ is an arbitrary functional.
  The aim is to solve
\begin{equation}\label{eq:prob_infty}
     \sup_{\sigma\in \Pi_\infty} V(\sigma), \mbox{ where } V(\sigma) :=H(F_\infty^\sigma).
\end{equation} 
In this situation, it is in general not 
true that an optimal policy can be found in 
the set of stationary deterministic 
Markovian policies as the following example 
shows. 

\begin{example}\label{exam2}
Consider the  degenerate Markov 
decision process given by: $E=\{0\}$, 
$A=\{0,1\}$ and $r(x,a)=\frac12 a,$ i.e., the decision alone determines the reward. 
Thus, the transition probabilities are given by $q(0|0,a)=1.$ Further assume that 
$\beta=1/2.$ This implies that $R_\infty$ is deterministic given by
$$ R_\infty = \frac12 \sum_{k=0}^\infty \left(\frac12\right)^k a_k,$$
where $(a_k)_{k=0}^\infty$ is the sequence of chosen 
actions. Hence, $R_\infty \in [0,1]$ and we 
can obtain any number in $[0,1]$ by 
choosing $(a_k)_{k=0}^\infty$ accordingly. Now fix  $B:=\{\sqrt{2}/2\}$ and let
$\sigma\in\Pi_\infty.$ 
For $F^\sigma_\infty\in P([0,1])$ define  the function $H(F^\sigma_\infty) := 
F^\sigma_\infty(B) =\PP^\sigma_{\delta_0}(R_\infty = \sqrt{2}/2).$ 
Clearly, $ \max_{\sigma\in\Pi_\infty} H(F^\sigma_\infty)=1$ for $\sigma^*$  
represented by the actions $(a_0,a_1\ldots )$ being the dyadic representation 
of $\sqrt{2}/2$. 
But $(a_0,a_1\ldots )$ is an infinite,  
non-periodic sequence. Thus, there is no hope 
for any optimal stationary  policy. Contrary, 
one can study the distribution of $R_\infty$ 
under the restriction of stationary 
deterministic Markovian policies. In such 
setting it makes sense to discuss the 
existence and properties of solutions to the 
distributional fixed point equation $R 
\stackrel{d}{=} X + \beta R,$ see e.g., 
\cite{neininger}. 
\end{example}

Here, we consider next the question of 
existence of optimal policies for 
\eqref{eq:prob_infty} and their approximation. 
We denote by $H=\cup_{n\in \N} H_n$ the set of all 
histories. Since $E$ and $A$ are finite, $H$ is 
countable. Therefore, $\Pi_\infty =P(A)^H=\{\sigma: 
H\to P(A)\}.$ By Tychonoff's theorem this set is 
compact in the product topology.  Note that
the mapping $\sigma\to \PP^\sigma_\nu$ is continuous when the set of strategic measures is equipped with the weak topology.



We proceed with the following assumption
\begin{description}
     \item[(C3)] The mapping $H$ is Lipschitz-continuous w.r.t.\ the Wasserstein $W_1$-distance, i.e., there exists $K>0$ such that for all $F,G\in P(\R)$
     $$ |H(F)-H(G)|\le K \cdot W_1(F,G).$$
\end{description}
Note that (C3) implies that $H$ is 
continuous w.r.t.\ weak convergence. This 
follows since weak convergence plus the 
convergence of the first moments is 
equivalent to convergence in the 
$W_1$-topology, see Theorem 6.9 in \cite{villani2008optimal}.
Define for $\sigma\in \Pi_\infty$ by 
$ F_N^\sigma$ the distribution of $R_N$ 
under $\PP^\sigma_\nu$:
$$F_N^\sigma=\PP^\sigma_\nu(R_N\in \cdot)$$
In this case, only the first $N$ decision rules of the sequence are important. 
Let 
\begin{equation}\label{eq:prob++}
    V_N(\sigma)  :=  H(F_N^\sigma), \quad   V_N := \sup_{\sigma\in \Pi_\infty} H(F_N^\sigma).
\end{equation} 
Further, let
\begin{eqnarray*}
    A_N &:=& \{ \sigma\in \Pi_\infty : V_N(\sigma) = \sup_{\sigma'\in \Pi_\infty} H(F_N^{\sigma'})\},\\
    A_\infty &:=&  \{ \sigma\in \Pi_\infty : V(\sigma) = \sup_{\sigma'\in \Pi_\infty} H(F_\infty^{\sigma'})\}.
\end{eqnarray*}
Note that under (C3) and our discussion the sets $A_N, A_\infty$ are non-empty. 
We obtain the following result where we  use the definition
$$ Ls A_N := \{\sigma\in \Pi_\infty : 
\sigma \mbox{ is an accumulation point of a 
sequence } (\sigma^N) \mbox{ with } 
\sigma^N \in A_N\}.$$

\begin{theorem}\label{theo:infty_main}
Under (C3) we obtain:
\begin{itemize}
\item[a)] The infinite horizon problem can 
be approximated by the finite horizon 
problems
$$ \lim_{N\to\infty} \sup_{\sigma\in\Pi_\infty} H(F_N^\sigma) = \sup_{\sigma\in\Pi_\infty} \lim_{N\to\infty}  H(F_N^\sigma) = \sup_{\sigma\in\Pi_\infty} H(F_\infty^\sigma).$$
\item[b)] There exists an optimal policy 
$\sigma^*=(\sigma_0^*,\sigma_1^*,\ldots)\in \Pi_\infty$ for problem 
\eqref{eq:prob_infty} and $\emptyset \neq 
Ls A_N \subset A_\infty.$
\end{itemize}
\end{theorem}

\begin{proof}
Parts a), b): Note that we have for fixed 
$\sigma\in \Pi_\infty$ and $n\ge m:$
$$ H(F_n^\sigma) - H(F_m^\sigma)\le K 
W_1(F_n^\sigma, F_m^\sigma) \le K 
\frac{\beta^{m+1} M }{1-\beta},$$
where  $\frac{\beta^{m+1} M}{1-\beta} \to 0$ for $m\to\infty.$
The latter inequality follows since
$$ W_1(F_n^\sigma, F_m^\sigma) \le \EE^\sigma_\nu |R_n-R_m| =  \EE^\sigma_\nu
\left(\sum_{k=m+1}^n \beta^k |r(X_k,A_k)|\right)\le \frac{\beta^{m+1} M}{1-\beta}.$$
Thus, the results are implied by 
Theorem A.1.5 in \cite{br11}. 
\end{proof}

 Theorem \ref{theo:infty_main} states that the value of the infinite horizon problem $\sup_{\sigma\in \Pi_\infty} H(F_\infty^\sigma)$ can be approximated by the value of the finite horizon problem $\sup_{\sigma\in \Pi_\infty} H(F_N^\sigma)$ for large $N$. This is not too surprising since $\beta\in (0,1)$ and the tail of the reward vanishes. Since the initial distribution $\nu$ if fixed, by  Theorem 2 in \cite{derman1966note}, the policy $\sigma^*$ can be replaced by a Markov policy in the original MDP model. 

\section{Special Cases and Applications}\label{sec:specialcases}
We now discuss some special choices for $H$ which are meaningful. In particular, by setting $H$ in the right way we rediscover some well-known results in the literature. In order to simplify the presentation, we use the finite state-action framework, i.e., we can work with sums instead of integrals.
\subsection{Classical MDP}
In the classical MDP theory, we want to maximize $\EE_{x_0}^\sigma R_N$ which is obtained when we choose $$H(F)=\sum_{x,s} (g(x)+s) F(x,s).$$
We show next that our general algorithm in the discrete case from \eqref{eq:OE} yields the established Bellman optimality equation as a special case. Let us briefly recall the standard theory and denote by $V_n : E\to \R$ the value function in the classical case.  Set $V_N=g$ and due to \eqref{eq:OEclassic0}
\begin{equation}\label{eq:OEclassic}
    V_n(x)= \max_a \big\{r(x,a) + \sum_{x'} 
    V_{n+1}(x')q(x'|x,a) \big\}.
\end{equation} 
The interpretation is 
$$V_n(x) = \sup_\sigma \EE_x^\sigma\Big[\sum_{k=n}^{N-1}r(X_k,A_k) + g(X_N)\Big],$$
i.e., $V_n$ is the maximal expected reward of 
the system from time $n$ onwards when we start 
at time $n$ in state $x.$ We begin with 
investigating our algorithm at time point 
$N-1:$
 \begin{align*}
J_{N-1}(F)  =& \sup_{\pi\in \Pi^M} H(T^\pi(F)) 
\\
 =& \sup_{\pi\in \Pi^M}  \sum_{x',s'} (g(x')+s') T^\pi(F)(x',s')\\
 =& \sup_{\pi\in \Pi^M}  \sum_{x',s'} 
 (g(x')+s')  \sum_{(x,s,a) :  r(x,a)=s'-s}  q(x'|x,a) \pi(a|x,s) F(x,s)\\
 =& \sup_{\pi\in \Pi^M}  \sum_{x'}  \sum_{(x,s,a) }  (g(x')+s+r(x,a)) q(x'|x,a) \pi(a|x,s) F(x,s)\\
   =& \sup_{\pi\in \Pi^M}\Big\{  \sum_{x'}  \sum_{(x,s,a) }  (g(x')+r(x,a))  q(x'|x,a) \pi(a|x,s) F(x,s) \\
   & +  \sum_{x'}   \sum_{(x,s,a) } s q(x'|x,a) \pi(a|x,s) F(x,s) \Big\}\\
    =&\sup_{\pi\in \Pi^M}  \sum_{x'}  \sum_{(x,s,a) }  (g(x')+r(x,a))  q(x'|x,a) \pi(a|x,s) F(x,s) +    \sum_{s } s  F(E,s)\\
    =& \sum_{s } s  F(E,s) +\sup_{\pi\in \Pi^M}  \sum_{x'}  \sum_{(x,s,a) }  (g(x')+r(x,a))  q(x'|x,a) F(x,s) \pi(a|x,s)    \\
     =& \sum_{s } s  F(E,s) +\sum_{x,s}F(x,s)\sup_{a}  \sum_{x'} (g(x')+r(x,a))  q(x'|x,a)     \\
 =&\sum_{s } s  F(E,s) +  \sum_{x}  
F(x,S)   \sup_{a} \sum_{x' }  (g(x')+r(x,a))  q(x'|x,a) \\
       =&\sum_{s } s  F(E,s) +  \sum_{x}  F(x,S)   \sup_{a} \{r(x,a)+\sum_{x' }  g(x')  q(x'|x,a)\}. 
 \end{align*}  
From this observation we obtain the following conjecture for time point $n$:
$$ J_n(F)=\sum_s s F(E,s) + \sum_{x} F(x,S) V_n(x),$$
where $V_n$ is the value function of the classical Bellman equation in \eqref{eq:OEclassic}. We prove this conjecture by
 induction over the time horizon. For $n=N$ we obtain:
\begin{align*}
J_N(F)= H(F) = \sum_{x,s} (g(x)+s) F(x,s) 
= \sum_{s} s F(E,s)+ \sum_{x} g(x) F(x,S).  
\end{align*}
For $n=N-1$ the statement is what we computed before. Now suppose the statement is true for $N,N-1,\ldots,n+1.$ We show that it is also true for time point $n.$
\begin{align*}
J_n(F)&=\sup_{\pi\in\Pi^M} J_{n+1}(T^\pi(F)) \\
&=  \sup_{\pi\in\Pi^M} \Big\{\sum_{s'} s' T^\pi(F)
(E,s') + \sum_{x'}  T^\pi(F)(x',S) V_{n+1}
(x')\Big\}\\
&=  \sup_{\pi\in\Pi^M} \Big\{\sum_{s',x'} s' 
T^\pi(F)(x',s') + \sum_{s',x'}  T^\pi(F)
(x',s') V_{n+1}(x') \Big\}\\
&=  \sup_{\pi\in\Pi^M}\Big\{ \sum_{s',x'} s' 
\sum_{(x,s,a) :  r(x,a)=s'-s}  q(x'|x,a) 
\pi(a|x,s) F(x,s) \\ & \quad  + \sum_{s',x'}  \sum_{(x,s,a) : r(x,a)=s'-s} q(x'|x,a) 
\pi(a|x,s) F(x,s) V_{n+1}(x')\Big\}\\
 &=  \sup_{\pi\in\Pi^M} \Big\{\sum_{x'}  \sum_{(x,s,a) } (r(x,a)+s) q(x'|x,a) \pi(a|x,s) F(x,s) \\  &\quad  + \sum_{x'}  \sum_{(x,s,a) }  q(x'|x,a) \pi(a|x,s) F(x,s) V_{n+1}(x') \Big\}\\
&=  \sup_{\pi\in\Pi^M} \Big\{ \sum_{(x,s,a) } r(x,a)  \pi(a|x,s) F(x,s) +  \sum_{(x,s) } s  
F(x,s)\\ & \quad  + \sum_{x'}  \sum_{(x,s,a) } 
q(x'|x,a) \pi(a|x,s) F(x,s) V_{n+1}(x')\Big\}\\
&=  \sum_{s} s  F(E,s)+ \sup_{\pi\in\Pi^M}  
\sum_{(x,s,a) }   \pi(a|x,s) F(x,s) \{r(x,a)+  
\sum_{x'}   q(x'|x,a)  V_{n+1}(x')\} \\
&=  \sum_{s} s  F(E,s)+   \sum_{(x,s) }   
F(x,s) \sup_a  \{r(x,a)+  \sum_{x'}   
q(x'|x,a)  V_{n+1}(x')\}  \\
&=  \sum_{s} sF(E,s)+\sum_{(x,s)}F(x,s) V_{n}(x) =  \sum_{s} s  F(E,s)+ \sum_{x }F(x,S) V_{n}(x).
\end{align*}
Thus, when we look at the last two lines, it is possible to recover the classical Bellman equation in our algorithm. In particular, the optimal strategy does not depend on $F$ and is deterministic.
The interpretation of $J_n(F_n)$  for $F_n=T^{\pi_{n-1}^*}\circ\ldots\circ T^{\pi_{0}^*}
(F_0)$ defined in Theorem \ref{theo:finite_main} 
with $F_0 = \delta_{x_0}\otimes\delta_0$ is as follows:
\begin{eqnarray*}
 J_n(F_n) &=& \sup_{(\sigma_0,\ldots,\sigma_{n-1})} \EE_{x_0}^{(\sigma_0,\ldots,\sigma_{n-1})}\Big[\sum_{k=0}^{n-1}r(X_k,A_k)\Big]+
\\
&&\sum_x F_n(x,S)  \sup_{(\sigma_n,\ldots,\sigma_{N-1})} \EE_{x}^{(\sigma_n,\ldots,\sigma_{N-1})}\Big[\sum_{k=n}^{N-1}r(X_k,A_k) + g(X_N)\Big].
\end{eqnarray*}
Hence, the algorithm separates at time point $n$ the future from the past.

\subsection{The case of quantile optimization}
 We are interested in the problem
 $$\sup_\sigma \PP^\sigma_\nu(R_{N}\ge t)$$
 for a fixed $t\in\R$. In other words, we want to maximize the probability that the accumulated reward exceeds the threshold $t.$ Hence, we can write the optimization criteria as  $$ \PP^\sigma_\nu(R_{N-1}+g(X_N)\ge t)= \sum_{x,s} F_N^\sigma(x,s) V_N(x,s),$$ where
 $$ V_N(x,s):= \left\{ \begin{array}{cl}
    1  & \mbox{ if } s+g(x)\ge t  \\
    0  & \mbox{ else. }
 \end{array}\right.$$
Problems like this have been investigated in \cite{filar1995percentile,wu1999minimizing,bauerle2011markov,chow2015risk,gilbert2017optimizing,li2022quantile}.
 It is known that the value for $\nu=\delta_{x_0}$ is given by $V_0(x_0,0)$ where $V_0$ can be computed from the recursion 
\begin{equation}\label{eq:quantile_rec}
    V_n(x,s) = \sup_a  \sum_{x'} V_{n+1}(x',s+r(x,a))q(x'|x,a), \quad n=0,1,\ldots,N-1.
\end{equation} 
Let us consider our algorithm again at time point $N-1:$
\begin{align*}
J_{N-1}(F) & = \sup_{\pi\in \Pi^M} H(T^\pi(F))  \\
&= \sup_{\pi\in \Pi^M}   \sum_{x',s'} 
V_N(x',s') T^\pi(F)(x',s')\\
&= \sup_{\pi\in \Pi^M}   \sum_{x',s'} 
V_N(x',s')   \sum_{(x,s,a) :  r(x,a)+s= s'} 
q(x'|x,a)  \pi(a|x,s) F(s,x)\\
&= \sup_{\pi\in \Pi^M}   \sum_{x'} 
\sum_{x,s,a} V_N(x',s+r(x,a))    q(x'|x,a)  
\pi(a|x,s) F(s,x)\\
&=  \sum_{x,s} F(s,x) \sup_{\pi\in \Pi^M}  
\sum_{x'}\sum_a  V_N(x',s+r(x,a))q(x'|x,a)  
\pi(a|x,s) \\
&=  \sum_{x,s} F(s,x) \sup_{a}   \sum_{x'} 
V_N(x',s+r(x,a))    q(x'|x,a) =:  
\sum_{x,s} F(s,x) V_{N-1}(x,s).
 \end{align*}  
This gives rise to the following conjecture:
$$ J_n(F)=\sum_{x,s}  F(x,s) V_{n}(x,s),$$
where $V_n$ is given in \eqref{eq:quantile_rec}.
Here we start with $V_N$ defined above.
For $n=N$, the statement is true by
definition of $V_N.$ Now suppose the 
statement is true for $N,N-1,\ldots,n+1.$ 
We show that it is also true for the time 
point $n.$
\begin{align*}
J_n(F)&= \sup_\pi J_{n+1}(T^\pi(F)) \\
&=  \sup_{\pi\in\Pi^M} \sum_{x',s'}  V_{n+1}(x',s')  
T^\pi(F)(x',s') \\
&=  \sup_{\pi\in\Pi^M} \sum_{s',x'}  V_{n+1}(x',s') 
\sum_{(x,s,a) :  r(x,a)=s'-s}  q(x'|x,a) 
\pi(a|x,s) F(x,s)\\
&=  \sup_{\pi\in \Pi^M}   \sum_{x'} 
\sum_{x,s,a} V_{n+1}(x',s+r(x,a))q(x'|x,a)
\pi(a|x,s) F(s,x)\\
&=  \sum_{x,s} F(s,x) \sup_{\pi\in \Pi^M}   
\sum_{x'}\sum_a  V_{n+1}(x',s+r(x,a))   
q(x'|x,a)  \pi(a|x,s) \\
&=  \sum_{x,s} F(s,x) \sup_{a}   \sum_{x'} 
V_{n+1}(x',s+r(x,a))    q(x'|x,a) = 
\sum_{x,s} F(s,x) V_{n}(x,s),
\end{align*}
which completes the induction step. Thus, 
we can see that the optimal strategy does not depend on $F$ and is deterministic. The 
recursion reduces to the recursion for 
$V_n$ which appears in a similar way in 
\cite{wu1999minimizing,bauerle2011markov}. 
The interpretation for $V_n$ is 
$$ V_n(x,s)= \sup_\sigma 
\PP^\sigma_\nu(R_{N-1}+g(X_N)\ge t| X_n=x, 
R_{n-1}=s).$$

 \subsection{Optimal transport}
We next consider a non-standard  application. We want to determine the transition probabilities of a  random walk in such a way that at a fixed time point a 
given distribution is best approximated in 
a certain metric, and at the same time the cost of transporting the probability mass is minimal.  The theoretical framework is that of Section \ref{sec:special} with some minor extensions. Instead of maximizing a reward,  we now formulate the problem as one of minimizing cost. This can be considered as an optimal transport problem. The optimal transport problem was  first discussed by \cite{monge1781memoire} and then refined by Kantorovich (1948), see \cite{kantorovich2006problem} for a reprint of the original article. A lot of researchers considered the optimal transport problem and many important results were established, see for example \cite{rachev2006mass2,rachev2006mass} or \cite{villani2008optimal} for  monographs on the topic. The optimal transport between discrete-time stochastic processes has connections to dynamic programming, see \cite{terpin2024dynamic,backhoff2017causal,moulos2021bicausal}. But we will pursue a different direction here. To be more specific, we 
consider the following data. Let $N\in\N$ be a fixed time horizon, $G$ be a given 
distribution on $E$. 
The terminal cost function  $H:P(E) \to \R$ is given by 
$$ H(F) := \int_{\R} |F_c(t)-G_c(t)|dt,$$
which is the Wasserstein $W_1$ distance between $G$ 
and the distribution $F$ of the terminal state 
$X_N$ of a controlled random walk. Here, $F_c$ and 
$G_c$ are cumulative distributions of $F$ and 
$G$, respectively.
More precisely, we define the approximating MDP (random walk) by

\begin{itemize}
\item[(i)] $E= \{1, \ldots ,K\},\, K\in\N,$  is the state space,
\item[(ii)] $A= \{(a^1,a^2)\in [0,1]^2 : a^1+a^2 \le 1  \}$
 is the  action space,
 \item[(iii)] the transition probability from $E\times A$ to $E$ for $x\in  \{ 2, \ldots ,K-1\}$ is given by
 $$ q(x'|x,a^1,a^2)= \left\{ \begin{array}{cc}
   a^1,   &  x'=x+1,\\
   a^2,   & x'=x-1,\\
   1-a^1 -a^2, & x'=x,
 \end{array}\right.$$
 and 
 $$ q(x'|0,a^1,a^2)= \left\{ \begin{array}{cc}
   a^1,   &  x'=1,\\
   1-a^1, & x'=0,
 \end{array}\right. \quad  q(x'|K,a^1,a^2)= \left\{ \begin{array}{cc}
   a^2,   &  x'=K-1,\\
   1-a^2 , & x'=K,
 \end{array}\right.$$
 \item[(iv)] $r_n(x,a^1,a^2)= c_n(a_1+a_2)$ where $0<c_0\le c_1 \le \ldots \le c_N\le 1$ are constants.
\end{itemize}
 The transition law is that for a random walk $(X_n)$ which can get from state $x$ only in state $x+1$ or $x-1$ or stay in $x$. The corresponding probabilities $a^1, a^2$ are subject to the decision. The choice of $a^1,a^2$ incurs some cost.
Note that we have a non-stationary problem here and the  transportation cost is non-decreasing in time. Thus, early transport is cheaper than later transport. The objective function is then
\begin{equation}\label{eq:prob_ex}
    \inf_{\sigma\in \Pi_N} \Big\{H(F_N^\sigma) +  \EE_\nu^\sigma [R_{N-1}]\Big\},
\end{equation} 
where $R_{N-1}=\sum_{k=0}^{N-1} c_k(A_k^1+A_k^2)$ with $A_k=(A_k^1,A_k^2).$ This is a weighted objective of distance to the desired distribution and expected transportation costs.
Note that the assumptions (C1) and (C2) of Section \ref{sec:model}  are satisfied and that the action space is compact. 
Thus, $(X_n)$ is a random walk on $E,$
where the up and down probabilities can be 
chosen and may depend on the history of the 
process. This kind of processes are called 
'elephant random walk' (see
\cite{gut2021variations}), 
since elephants have a long memory. In line 
with Section \ref{sec:special}, it is here 
enough to consider the recursion for the 
distribution of the first component $X_n, $ 
hence we write 
$F_n^\sigma(x)=\PP^\sigma_\nu(X_n=x),\, x\in E.$   Recall that $F_0^\sigma=\nu.$
Since the costs and transition functions have special structure, i.e., they are linear in the action variables, it is sufficient to consider
the set of non-randomized actions in the lifted 
MDP, that is, the set of all mappings 
$\pi :E\to A.$ With a little abuse of notation we
denote the set by the same symbol $\Pi^M$ as in the previous sections.
By a slight extension of Theorem \ref{theo:finite_spec}, the value functions in this application are for $n=0,1,\ldots ,N-1$ given by
\begin{align}
  \nonumber  J_N(F) & := H(F),\\
    J_n(F) &:=  \inf_{\pi\in \Pi^M} \left\{ \hat r_n (F,\pi)+ J_{n+1}(T^\pi(F)) \right\}, \label{eq:OE2}
\end{align}
where 
\begin{eqnarray*}
    \hat r_n(F,\pi) &=& \sum_x  r_n(x,\pi(x))  F(x) \\
    &=& c_n \sum_x F(x) (\bar{a}^1(x)+\bar{a}^2(x)),
\end{eqnarray*}
with $\pi(x)=(\bar{a}^1(x),\bar{a}^2(x)).$ In particular, $\pi(1)=(1,0)$ and $\pi(K)=(0,1)$ (i.e., $\bar{a}^2(1)=\bar{a}^1(K)=0$) and
\begin{eqnarray*}
T^\pi (F)(x') &=&
\sum_x q(x'|x,\pi(x)) F(x)=
\sum_x q(x'|x,\bar a^1(x),\bar a^2(x)) F(x)\\
&=& \bar{a}^1(x'-1) F(x'-1) + \bar{a}^2(x'+1) F(x'+1) + (1- \bar{a}^1(x')-\bar{a}^2(x')) F(x').
\end{eqnarray*} 
 We start with a given 
distribution $F_0$ on $E$ and have to compute 
$J_0(F_0)$.  The value function $J_n(F)$ describes the minimal sum of  the expected transportation cost and remaining distance of a  terminal distribution of the random walk to the target, starting at time $n$ till the terminal time $N.$

\begin{remark}[Connection to classical optimal transport]
 Note that when  all 
 $c_k\equiv  1,$ then $J_0(F_0)$ is exactly the 
Wasserstein distance to the distribution $G$, 
 also known as the earth-mover distance.  In order to see this, recall that the Wasserstein $W_1$ distance between two discrete distributions
   \[
        \mu=\sum_{j=1}^{K}\omega_{j}\delta_{j} \quad \text{ and }\quad  \nu=\sum_{i=1}^{K}\vartheta_i\delta_{i},
    \]
    where $\omega_j, \vartheta_i \geq 0 $, for all $i,j=1,\ldots,K$   and $\sum_{j=1}^{K}\omega_j=\sum_{i=1}^{K}\vartheta_i=1$ can be computed from the following linear program (see \cite{kantorovich2006problem})

   \begin{align}\label{Prob:LP}\tag{P}
    \min \quad   &\sum_{j=1}^{K} \sum_{i=1}^{K}q_{j,i}|j-i|
    \\ \notag
    \text{s.t.} \quad  &\sum_{i=1}^{K}q_{j,i}=\omega_j, \quad  j=1,\ldots, K ,
    \\ \notag
   &\sum_{j=1}^{K}q_{j,i}=\vartheta_i, \quad  i=1,\ldots,K,
    \\ \notag
    &   q_{j,i} \geq 0, \quad j=1,\ldots,K, \ i=1,\ldots, K.
 \end{align} 
This linear program can be interpreted as a problem of transporting probability mass from the  distribution $\mu$ to the distribution $\nu$ with
minimal cost, where the cost is given as the sum of 
single masses times transportation distance. In our problem formulation we can - within one step - only transport mass from one state to its neighbors, i.e., from 4 to 3 and 5, but not to 6. However, obviously for a given transport $q=q_{j,i}>0$ in $(P)$ with $j<i$ we can split the transport in $i-j$ transports of the form $q_{j,j+1}=q, q_{j+1,j+2}=q,\ldots, q_{i-1,i}=q. $ The cost will be identical, because $q_{j,i}|j-i| = q_{j,j+1}+q_{j+1,j+2}+\ldots + q_{i-1,i}.$ Thus, consider for a moment only transports with the additional restriction $q_{j,i}=0$ for $i\notin \{j+1,j-1\}.$ The corresponding cost of such a transport can be expressed by $\sum_{j=0}^{K-1}q_{j,j+1} + \sum_{j=1}^K q_{j,j-1}.$
Now consider one time step of our model, say from $n$ to $n+1$ and let $X_n\sim\mu$ and $X_{n+1}\sim \nu.$   We can realize the same transport in our model by choosing $\bar{a}^1(j)= \frac{q_{j,j+1}}{\omega_j}$ and $\bar{a}^2(j)= \frac{q_{j,j-1}}{\omega_j}$. Note that the (conditional) transition probabilities implied by a transport $(q_{j,i})$ in $(P)$ are given by $\PP(X_{n+1}=i|X_n=j)= \frac{q_{j,i}}{w_j}.$ The one-step cost of the transport if $c_k\equiv 1$ is $$\EE(A_n^1+A_n^2)=\sum_{j=1}^K \omega_j (\frac{q_{j,j+1}}{\omega_j} + \frac{q_{j,j-1}}{\omega_j})= \sum_{j=0}^{K-1}q_{j,j+1} + \sum_{j=1}^K q_{j,j-1}$$
and hence the cost are the same in both models. Thus,  the optimal transport of $(P)$ can be implemented by decomposing it in neighboring transports and the terminal cost given by the Wasserstein distance to the target.  The remaining cost is exactly given by the Wasserstein distance in the end.

However, this argument breaks if  $c_k\not\equiv 1$. In our problem formulation there is a time aspect: It is cheaper to move 
 mass early. In particular if $K=4$ and 
 $F_0=(\frac12,0,0,\frac12)$ and $G=(\frac12,
\frac12,0,0)$ and $N\ge 2$, then it is easy 
to see that $W_1(F_0,G)=1$ and $J_0(F_0)= 
\frac12(c_1+c_2)$. However, if   $F_0=(0,
\frac12,0,\frac12)$ and $G=(\frac12,0,
\frac12,0)$  then  again $W_1(F_0,G)=1$, but 
$J_0(F_0)= c_1$. Hence, we get the same $W_1$-
distance but different results in our 
optimization problem.
\end{remark}

\begin{theorem}\label{theo:OT}
The following statements hold  for the optimal policy:
\begin{itemize}
\item[a)] If at time point $n$  not all mass 
at $x$ is moved, then $x$w ill only receive 
mass at later time points and stay a sink.
Formally, for any $x\in E$  
$$ \bar a_n^1(x)+\bar a_n^2(x)< 1 \quad \Rightarrow\quad \bar a_m^1(x)+\bar a_m^2(x)=0,\, m\ge  n+1. $$
\item[b)]  A mass which has already been moved 
will keep its direction of moving.
    \end{itemize}
\end{theorem}

\begin{proof} Part a): 
Assume that $\bar a_n^1(x)+\bar a_n^2(x)< 1$
but $\bar a_m^1(x)+\bar a_m^2(x)>0$
for some $m\ge n+1.$ Note that the cost from the coordinate $x$ equals
$$c_nF_n(x)(\bar a_n^1(x)+\bar a_n^2(x))$$
and $F_{n+1}(x)\ge F_n(x)(1-\bar a_n^1(x)-\bar a_n^2(x)).$   Let $m\ge n+1$ be the smallest  number for which
$a_m^1(x)+a^2_m(x)>0,$ which implies that $F_m(x)\ge F_{n+1}(x).$ 
Suppose that it is optimal in period $m$ is to transfer at least the mass $F_n(x)(\alpha_1\bar a_n^1(x)+\alpha_2\bar a_n^2(x)),$ where $\alpha_1,\alpha_2\in [0,1)$ are numbers such that 
$$0<\alpha_1\bar a_n^1(x)+\alpha_2\bar a_n^2(x)\le 1-
\bar a_n^1(x)-\bar a_n^2(x).$$
The cost from $x$ in period $m$ is 
$$c_mF_m(x)(\bar a_m^1(x)+\bar a_m^2(x))\ge
c_m F_n(x)(\alpha_1\bar a_n^1(x)+\alpha_2\bar a_n^2(x)).
$$
But if the mass were transported at $n,$ then the cost would be 
 $$c_n F_n(x)(\alpha_1 \bar a_n^1(x)+\alpha_2 \bar a_n^2(x))\le c_m  F_n(x)(\alpha_1\bar a_n^1(x)+\alpha_2\bar a_n^2(x)).$$
 Hence, it is not optimal  to transfer the mass 
 $F_n(x)(\alpha_1 \bar a_n^1(x)+\alpha_2 \bar a_n^2(x))$ at period $m.$ 
 It is cheaper to transfer it at period $n$. Then, 
 $\bar a^1_n(x)$ and $\bar a^2_n(x)$ would not be optimal.
 Therefore, it must hold $a_m^1(x)=a^2_m(x)=0$ for 
 $m\ge n+1.$

 Even if $x$ receives new mass at later time 
points this will not be moved, because without 
loss of generality we can assume that older 
mass is moved first.


Part b): Moving a mass in one direction and 
back at a later time results in the same 
distribution and just produces costs. Hence, 
this cannot be optimal.
\end{proof}

Note. that the optimal policy does not depend on the precise value of the variables $c_n.$ However, the assumption $0<c_0\le c_1 \le \ldots \le c_N\le 1$ implies the structural properties in Theorem \ref{theo:OT}.
From these observations, we obtain the 
following algorithm: Suppose the state 
(distribution) at time $n$ is given by 
$(F_n(1),\ldots,F_n(K))$ and $(G(1),
\ldots,G(K))$ is the probability mass function 
of $G$. Define at point $x\in E:$
$$ \Delta F_n^{l}(x) := \sum_{j=1}^{x-1} 
F_n(j) - \sum_{j=1}^{x-1} G(j), \quad   \Delta 
F_n^{u}(x) := \sum_{j=x+1}^{K} F_n(j) - 
\sum_{j=x+1}^{K} G(j).$$
If $\Delta F_n^{l}(x) <0$, then in comparison 
to $G$ mass is missing left of $x$, if $\Delta 
F_n^{l}(x) =0,$ then  in comparison to $G$ the 
mass left of $x$ is sufficient, if $\Delta 
F_n^{l}(x) >0$, then in comparison to $G$ 
there is too much mass  left of $x$. The 
expression $\Delta F_n^{u}(x)$ has a similar 
explanation concerning the mass right of $x.$

The optimal amount of  mass which is moved can 
now be determined locally  at each point $x\in 
E$, just by knowing $\Delta F_n^{l}(x)$ and $
\Delta F_n^{u}(x).$ At every stage $n$, we 
have to go through all points $x\in E$ and do 
the following:

\begin{algorithm}
\caption{Optimal mass transportation algorithm}\label{alg:cap}
\begin{algorithmic}
\Require $n \geq 0$ and $(F_n(1),\ldots,F_n(K))$ 
\While{$n < N$} for all $x\in E$
\If{ $\Delta F_n^{l}(x)\ge0$ and $\Delta F_n^{u}(x)\ge0$}
    \State $a_n^1(x)=a_n^2(x)=0.$
\ElsIf{ $\Delta F_n^{l}(x)\ge0$ and $\Delta F_n^{u}(x)<0$}
 \State $a_n^1(x)=\min\{ F_n(x), -\Delta F_n^{u}(x)\}$.
 \State $a_n^2(x)=0$
\ElsIf{$\Delta F_n^{l}(x)<0$ and $\Delta F_n^{u}(x)\ge0$}
    \State $a_n^2(x)=\min\{ F_n(x), -\Delta F_n^{l}(x)\}$
     \State $a_n^1(x)=0$
\ElsIf{$\Delta F_n^{l}(x)<0$ and $\Delta F_n^{u}(x)<0$}
    \State $a_n^1(x)= -\Delta F_n^{u}(x)$
    \State  $ a_n^2 (x)=-\Delta F_n^{l}(x)$.
\EndIf
\State Update $F_{n+1}$
\State $n \gets n + 1$
\EndWhile
\end{algorithmic}
\end{algorithm}

In the situation of the last case ($\Delta F_n^{l}(x)<0$ and $\Delta F_n^{u}(x)<0$) it follows that $F_{n+1}(x) = G(x)$ and there  automatically will be enough mass at point $x$ to realize the move.

Also note that this situation can be extended to approximating a stochastic process in continuous time by a Markov chain such that the distance between the marginal distributions is minimized.

Let us conclude this special case by illustrating some numerical results obtained from implementing Algorithm~\ref{alg:cap}. First, we used a rescaled normal distribution on $E=\{1,\ldots,K\}$ as the target distribution, i.e., $G = (G(1),\ldots,G(K))$ for $G(j) = \varphi_{K/2,\sigma^2}(j)/\sum_{\ell = 1}^K \varphi_{K/2,\sigma^2}(\ell)$, where $\varphi_{K/2,\sigma^2}$ denotes the density function of a normal distribution with mean $K/2$ and variance $\sigma^2$, where we chose $\sigma\in \{0.5,1,2,5\}$. For the initial distribution, we used $m=100$ samples of probability distributions on $E$ obtained from sampling $K$ i.i.d. random variables uniformly distributed on $\{0,1,2,\ldots,10\}$ and appropriate scaling. For $K\in \{50,100\}$, we applied the algorithm for $N\in \{1,\ldots,K/2+5\}$ and determined the average Wasserstein distance of the resulting distribution to the target distribution $G$ over the $m=100$ sampled initial distributions. This seems more interesting and informative than the value $J_0$ of the optimization problem itself, since this is difficult to interpret.  The results are shown in Figure~\ref{fig:normal}. We observe a linear decrease of the Wasserstein distance in terms of the number of time steps $N$. Depending on the choice of $\sigma$, we notice some index $N_0 = N_0(\sigma)$ such that $H(F_N) \approx 0$ for $N\geq N_0$. It appears to be decreasing in terms of the standard deviation $\sigma$, but even for the smallest choice of $\sigma = 0.5$, only $K/2$ time steps are necessary to achieve an accurate approximation of the target distribution. Note that using $N\geq K-1$ time steps allows us to achieve any target distribution starting from any initial distribution due to the construction of the algorithm with the 'worst case' (i.e., the case with a necessary number of $N=K-1$ time steps) being $F_0 = (1,0,\ldots,0)$ and $G=(0,\ldots,0,1)$.

Figure~\ref{fig:boxplot} additionally shows boxplots generated from the Wasserstein distances of the target distribution $G$ for $K\in \{50,100\}$ and $\sigma \in \{0.5,1,2,5\}$ and the distribution $F_{30}$ obtained after performing the first $30$ steps of Algorithm~\ref{alg:cap} for the $m=100$ initial samples. We can again observe the smallest Wasserstein distance for the largest value of $\sigma$. Moreover, we notice that the dispersion of the Wasserstein distances is smallest for the largest $\sigma$ as well.

\begin{figure}[!t]
\centering  
\subfigure[$K=50$]{\label{fig:normal50}\includegraphics[width=0.48\textwidth]{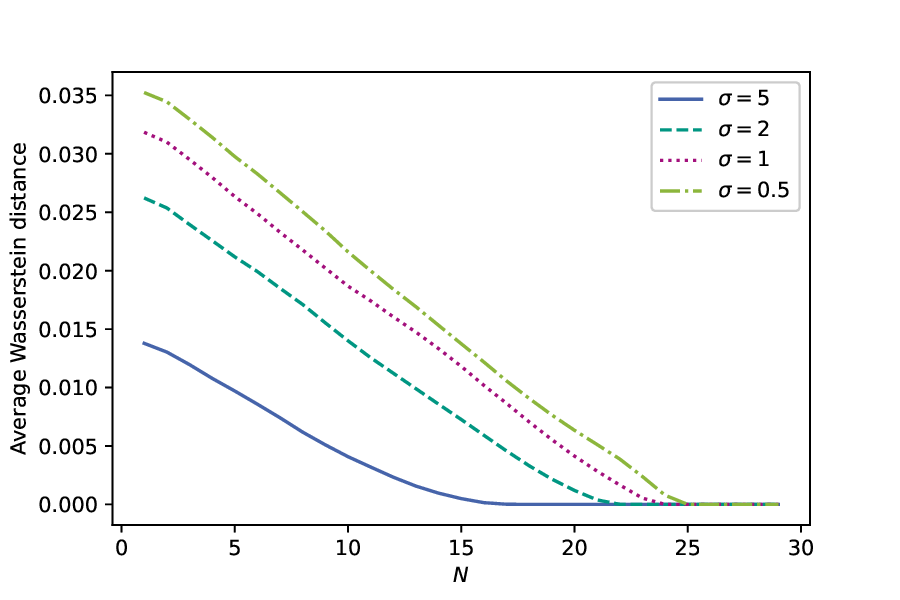}}
\subfigure[$K=100$]{\label{fig:normal100}\includegraphics[width=0.48\textwidth]{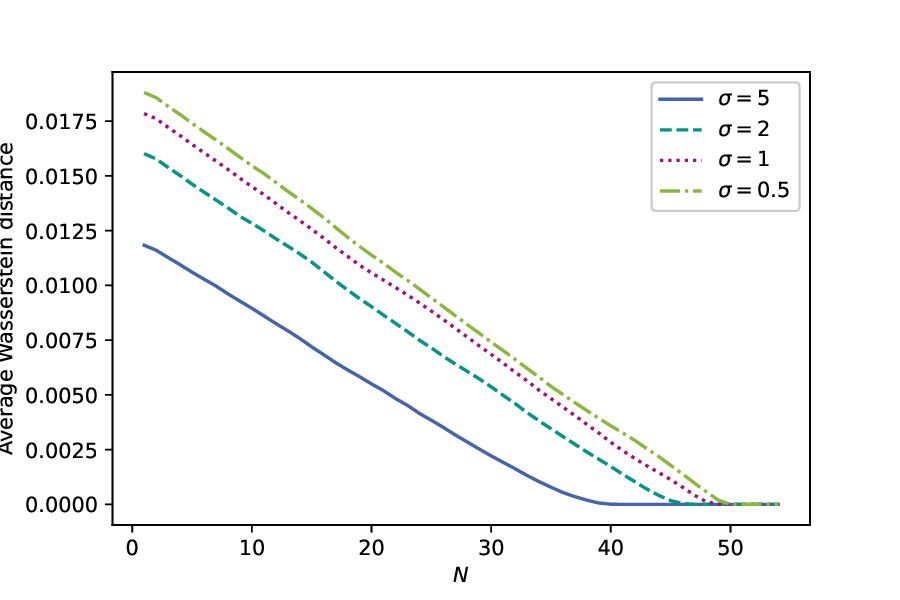}}
\caption{Average Wasserstein distance of $F_N$ from Algorithm~\ref{alg:cap} and the target distribution $G$ (rescaled normal distribution) over $m=100$ initial samples for $\sigma\in \{0.5,1,2,5\}$ and $K\in \{50,100\}$.}
\label{fig:normal}
\end{figure}

\begin{figure}[!t]
\centering  
\subfigure[$K=50,\, n = 15$]{\label{fig:box50}\includegraphics[width=0.48\textwidth]{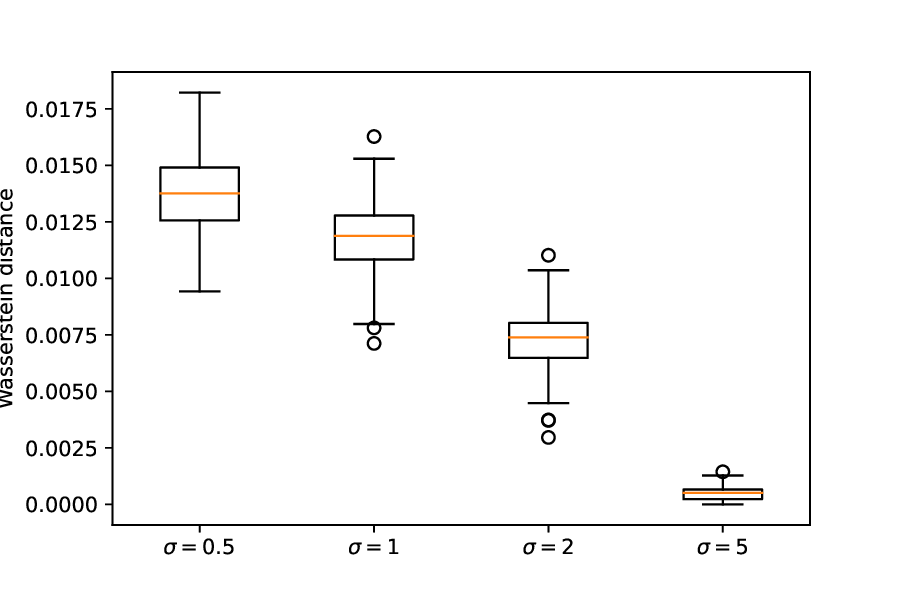}}
\subfigure[$K=100,\, n = 30$]{\label{fig:box100}\includegraphics[width=0.48\textwidth]{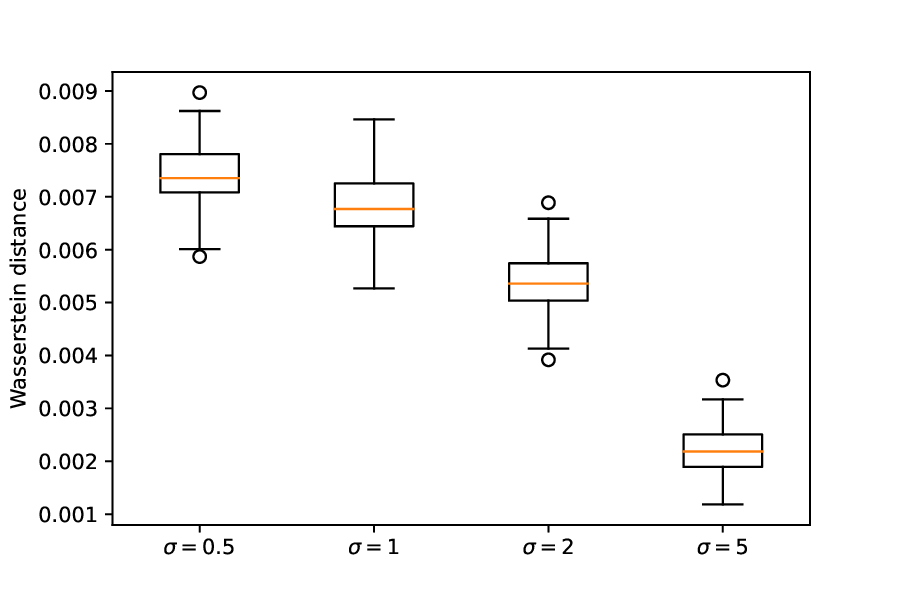}}
\caption{Boxplots of the Wasserstein distance of $F_n$ from Algorithm~\ref{alg:cap} and the target distribution $G$ (rescaled normal distribution) using the data from the $m=100$ initial samples for $\sigma\in \{0.5,1,2,5\}$.}
\label{fig:boxplot}
\end{figure}


\begin{figure}[!ht]
\centering  
\subfigure[$K=50$]{\label{fig:exp50}\includegraphics[width=0.48\textwidth]{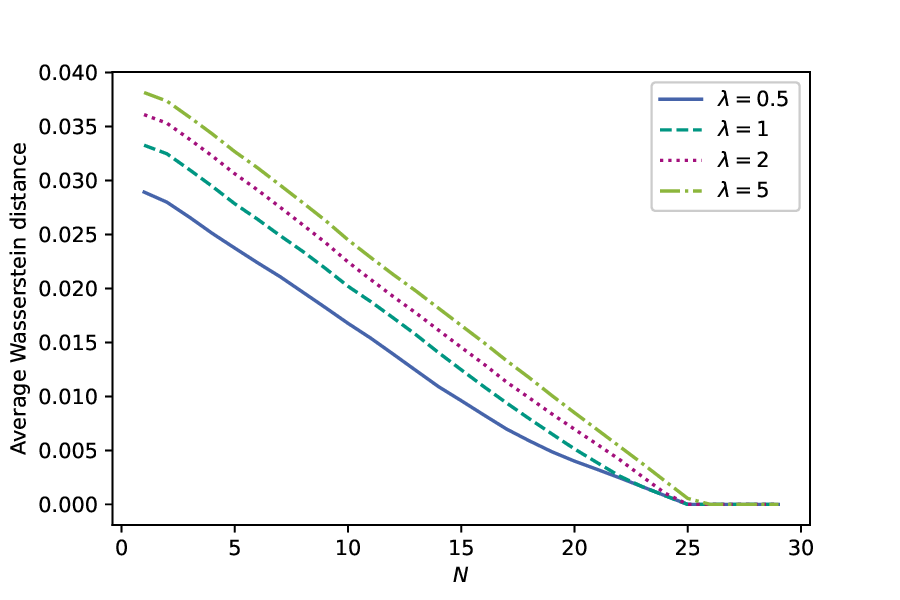}}
\subfigure[$K=100$]{\label{fig:exp100}\includegraphics[width=0.48\textwidth]{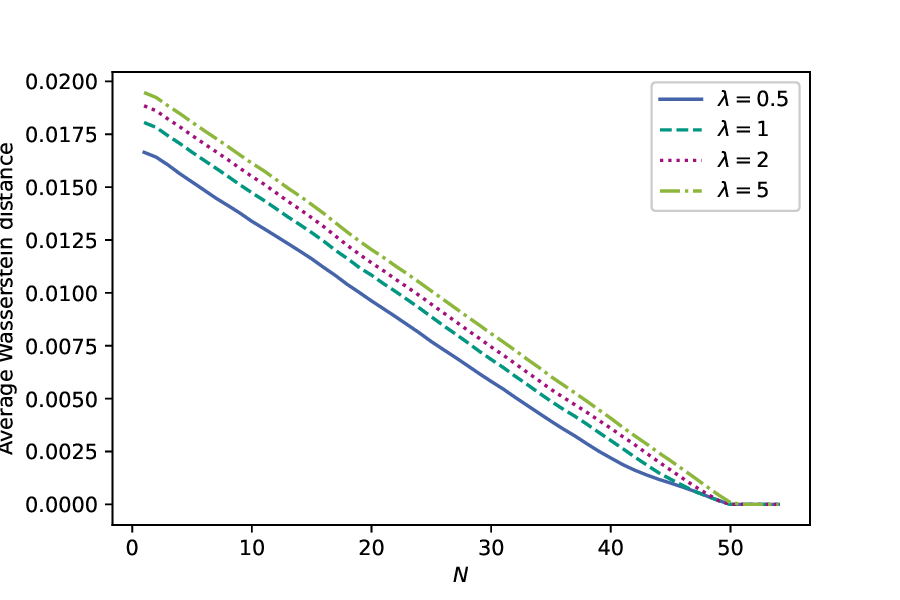}}
\caption{Average Wasserstein distance of $F_N$ from Algorithm~\ref{alg:cap} and the target distribution $G$ (rescaled shifted exponential distribution) over $m=100$ initial samples for $\lambda\in \{0.5,1,2,5\}$ and $K\in \{50,100\}$.}
\label{fig:exponential}
\end{figure}

For a second example, we used a rescaled shifted exponential distribution as the target distribution, i.e., $G = (G(1),\ldots,G(K))$ for $G(j) = f_\lambda(j)/\sum_{\ell = 1}^K f_\lambda(\ell)$, where $f_\lambda(x) = \lambda \mathrm{e}^{-\lambda (x-K/2)}$, $x>K/2$, denotes the density of an exponential distribution with parameter $\lambda > 0$ which was shifted to the right by $K/2$. We chose $\lambda \in \{0.5,1,2,5\}$ for our numerical example. The initial distributions are sampled in the same way as before. The resulting average Wasserstein distance of $F_N$ and $G$ over the $m=100$ initial samples $F_0$ is shown in Figure~\ref{fig:exponential}. Similarly to the case of a rescaled normal target distribution, we notice a linear decay of the average Wasserstein distance. Moreover, the index $N_0 = N_0(\lambda)$ with $H(F_N) \approx 0$ for $N\geq N_0$ seems to be located at $K/2$ for any choice of $\lambda$. However, for $N<N_0$, the average Wasserstein distance is larger for larger values of $\lambda$.

\subsection{Further Applications}
The approach we describe here offers maximal flexibility in shaping reward distributions to fit some targeted distributions. This is for example interesting in portfolio optimization where constraints have to be met (e.g., Value-at-Risk constraints, see \cite{basak2001value}) or the return distribution is shaped according to the risk preferences of the agent. There have been some applied studies in this direction \cite{brayman2023profile} and some theoretical  findings about the connection between target criterion and trading strategy \cite{cox2014utility} as well as solution techniques \cite{he2011portfolio}. This aspect could be interesting for personalized robo-advising \cite{capponi2022personalized} where AI tools choose portfolio strategies which match with the agents' preferences.

Another application, as indicated in the last subsection, is an optimal approximation of a continuous-time stochastic process $(X_t)$ (e.g. Brownian motion) by a discrete-time Markov chain $(\hat X_n)$ in the sense that at pre-determined time points $t_1 < t_2<\ldots < t_n$ the weighted distance $\sum_{j=1}^n\alpha_j W_1(X_{t_j},\hat X_{j})$ for $\alpha_j \ge 0$ has to be minimized. In this situation the transition probabilities of the discrete-time Markov chain have to be chosen. There may be constraints about the domain of the transition kernel, i.e., 
transitions to only certain values may be feasible.

More demanding applications include for example the optimal control of crowd behavior. Such models  are often formulated in terms of mean-field systems \cite{carmona2018probabilistic} where the individual behavior is modeled in relation to the empirical distribution of the other individuals. Individuals might be persons, cars, animals and so on. Often a certain distribution of those individuals is preferred for example to avoid congestion, to guarantee smooth exiting of a building, to obtain an optimal shape of a flock of birds etc. (see Sect. 1 in \cite{carmona2018probabilistic}). Thus, the aim is here to find a policy such that the distribution of the individuals is optimally shaped.

\section{Conclusion}
In this paper, we studied Markov Decision Processes where the objective consists of functionals of the distribution of the accumulated reward. We showed that these kind of problems can be formulated as a dynamic program by defining a lifted MDP. The corresponding Bellman equation yields an algorithm for solving these problems. We have seen that this approach comprises many well-studied problems and allows a different point of view on the traditional Bellman equation.

\bibliographystyle{apalike}
\bibliography{literature_DBE}
\end{document}